\documentclass[11pt]{article}
\usepackage{amsfonts,amsmath}
\usepackage{latexsym}
\usepackage{amsthm}
\usepackage{amscd}
\usepackage{commath}
\usepackage{epsfig}
\usepackage{amssymb}
\usepackage{caption}
\usepackage{enumitem}
\usepackage{mathtools}
\usepackage{booktabs}
\usepackage[table]{xcolor}
\usepackage{authblk}
\usepackage{array}
\newcolumntype{P}[1]{>{\centering\arraybackslash}p{#1}}

\usepackage[english]{babel}

\usepackage[letterpaper,top=2cm,bottom=2cm,left=2.5cm,right=2.5cm,marginparwidth=1.5cm]{geometry}

\usepackage[colorlinks=true, allcolors=blue,bookmarksnumbered]{hyperref}
\hypersetup{
	citecolor = red!80!black
}
\usepackage{color,soul}
\setstcolor{red}
\usepackage{float,subcaption}
\usepackage{graphicx}
\usepackage{stackengine}
\usepackage{tikz}
\usetikzlibrary{intersections}
\usepackage[export]{adjustbox}
\usepackage{array}

\usepackage[capitalise]{cleveref}
\usepackage{microtype}
\usepackage[percent]{overpic}

\usepackage{lineno}

\newtheorem{defn}{Definition}[section]
\newtheorem{thm}[defn]{Theorem}

\newtheorem{prop}[defn]{Proposition}
\newtheorem{lem}[defn]{Lemma}
\newtheorem{cor}[defn]{Corollary}
\theoremstyle{remark}
\newtheorem{remark}[defn]{Remark}

\numberwithin{equation}{section}
\numberwithin{figure}{section}

\newcommand{\bb}{\begin{equation}}
\newcommand{\ee}{\end{equation}}

\newcommand{\matteo}[1]{{\small \color{green!80!black} [Matteo: #1]}}
\newcommand{\eps}{\varepsilon}

\newcommand{\origin}{\mathbf{o}}

\newcolumntype{?}{!{\vrule width 1.5pt }}
\newlength\savedwidth

\newcommand{\tightoverset}[2]{%
  \mathop{#2}\limits^{\vbox to -.5ex{\kern-0.75ex\hbox{$#1$}\vss}}}

\newenvironment{proofof}[1]{{\medbreak\noindent \em Proof of #1.\enspace }}{\hfill\qed\medbreak}
\newcommand\sut{\,;\;}
\newcommand{\dfn}[1]{\textbf{\textit{#1}}}
\newcommand{\MD}[1]{{{\color{green!50!black} [Matteo: #1]}}}
\newcommand{\tbc}[1]{{{\color{orange}  #1}}}
\newcommand\mob{\textnormal{\textsf{M\"ob}}}  
\newcommand\HH{\mathbb{H}} 
\newcommand\Unif{\mathrm{Unif}}  

\newcommand\Leb{\mathrm{Leb}}

\newcommand{\bud}{\mathbf{d}}  

\newcommand\ST{\,;\;}
\newcommand\Kh{\widehat K}  

\newcommand\II[1]{\mathbf{1}_{#1}}  
\newcommand\Vor{\mathrm{Vor}}  

\def\rlabel #1 #2{\begin{equation} \label{#1} #2 \end{equation}}

\def\rproof{\begin{proof}}

\def\Qed{\end{proof}}

\def\rcases#1{\begin{cases} #1 \end{cases}}

\title{\textsc{Ideal Poisson--Voronoi tessellations \\ beyond hyperbolic spaces}}
\author{Matteo \textsc{D'Achille}}
\affil{Laboratoire de Mathématiques d’Orsay, CNRS, Université Paris-Saclay, 91405, Orsay, France}

\date{\today}

\begin{document}

\maketitle

\begin{abstract}
We construct and study the ideal Poisson--Voronoi tessellation of the product of two hyperbolic planes $\mathbb{H}_{2}\times \mathbb{H}_{2}$ endowed with the $L^{1}$ norm. We prove that its law is invariant under all isometries of this space and study some geometric features of its cells. Among other things, we prove that the set of points at equal separation to any two corona points is unbounded almost surely. This is analogous to a recent result of Fr\c{a}czyk--Mellick--Wilkens for higher rank symmetric spaces.
\end{abstract}

\section{Introduction}

The main purpose of this paper is to demonstrate a further use of our abstract deterministic Theorem~\cite[Theorem 2.3]{IPVT23}. This Theorem gives sufficient conditions for the convergence of Voronoi diagrams in any boundedly compact metric space $(E,d)$, such as any complete locally compact length space by the Hopf--Rinow--Cohn--Vossen Theorem~\cite[Theorem 2.5.28]{BBI}.

\medskip
If the nuclei of the Voronoi diagram are a Poisson point process (PPP) of vanishing intensity, a random non-trivial tessellation may appear termed \dfn{ideal Poisson--Voronoi tessellation} (${\rm IPVT}$). Then~\cite[Theorem 2.3]{IPVT23} provides the basic deterministic recipe for studying IPVTs as low-intensity limits: first, study the convergence of the nuclei to the Gromov boundary of $E$ (in the sense of~\cite{Gro81}); second, identify (proto)-delays from large metric balls. 

\medskip
In~\cite{IPVT23}, this recipe has been thoroughly illustrated for the IPVT of real hyperbolic space $\HH_{d}$, $d\geq 2$ and, to a lesser extent, for the $k$-regular tree $\mathbb{T}_{k}$, $k\geq 3$. Earlier investigations of low-intensity Poisson--Voronoi and Bernoulli--Voronoi tessellations respectively on $\HH_{2}$ and $\mathbb{T}_{k}$ can be found in the PhD thesis of Bhupatiraju~\cite{bhupatiraju}.
 Budzinski--Curien--Petri~\cite{BudzinskiCurienPetri} used ${\rm IPVT}(\HH_{2})$, called there ``pointless Poisson--Voronoi tessellation'', to bound from above the Cheeger constant of hyperbolic surfaces in large genus. Fr\c{a}czyk--Mellick--Wilkens~\cite{FMW} used ${\rm IPVT}(\mathcal{X})$, where $\mathcal{X}$ is either a higher rank semisimple real Lie group or the product of at least two automorphism groups of regular trees, to prove that such $\mathcal{X}$ have fixed price $1$, thus answering positively a question of Gaboriau~\cite{gaboriau2000} in these cases. Mellick~\cite{mellick2024} recently addressed indistinguishability of the cells of ${\rm IPVT}(\mathcal{X})$ for such $\mathcal{X}$, after a question that we raised in~\cite[Question 7.8]{IPVT23}.

\clearpage
In this paper we construct and study ${\rm IPVT}(\mathcal{M})$, where $\mathcal{M}$ is the Cartesian product $\HH_{2} \times \HH_{2}$ endowed with the $L^{1}$ metric\footnote{We could have considered the Cartesian $m$-product $(\HH_{d_{1}} \times \HH_{d_{2}}\cdots \times \HH_{d_{m}},L^{1})$, with $d_{i}\geq 2$ for all $i=1,\ldots,m$. However we preferred to stick to the case $m=2$ and $d_{1}=d_{2}=2$ for the sake of clarity.}. 
Notice that $\mathcal{M}$ is neither a $\delta$-hyperbolic space for any $\delta >0$ (due to the presence of flats) nor a symmetric space (since it doesn't support a Riemannian metric). Nevertheless, our choice of {endowing $\HH_{2} \times \HH_{2}$ with the} $L^{1}$ metric and usual properties of Poisson point processes induce an appealing product structure with the following consequence:
 {
 \begin{thm}[\textsc{Convergence towards IPVT($\mathcal{M}$), short version}]\label{thm.conv}
 
  Let $\mathbf{X}^{(\lambda)}=(X_i^{(\lambda)} \, \ST \, i\geq 1 )$ be a PPP with intensity $\lambda \cdot \mathrm{Vol}_{ \mathcal{M}}$, with $\lambda>0$. Let $\Vor(\mathbf{X}^{(\lambda)})$ be the Voronoi diagram relative to $\mathbf{X}^{(\lambda)}$.  Then the following convergence in law holds
  $$
  \Vor(\mathbf{X}^{(\lambda)}) \underset{\lambda \downarrow 0}{\overset{\rm law}{\Longrightarrow}} {\rm IPVT}(\mathcal{M}) \; ,
  $$
 where ${\rm IPVT}(\mathcal{M})$ denotes the \dfn{ideal Poisson--Voronoi tessellation} of $\mathcal{M}$.
 \end{thm}
 
 \smallskip
Like ${\rm IPVT}(\mathbb{H}_{d})$, ${\rm IPVT}(\mathcal{M})$ is constructed from a PPP on the \dfn{corona} $\widetilde{\partial \mathcal{M}}$, which is here $\partial \mathbb{H}_{2} \times \partial \mathbb{H}_{2}$ cross $\mathbb{R}_{\geq 0}$, where $\partial \mathbb{H}_{2}$ denotes the Gromov boundary of $\mathbb{H}_{2}$ (see \Cref{f.corona} for a portrait). Informally, this holds because the nuclei are escaping at infinity at roughly the same speed in the two factors. The corona process still enjoys the factor property with a simple corona measure $\mu$, see \Cref{prop.convipvt} (which is the extended version of \Cref{thm.conv}) for details. This is reflected on simple product formulas for the exponential separation $\bud:\mathcal{M}\times  \widetilde{\partial \mathcal{M}} \to \mathbb{R}_{\geq 0}$ to any corona point (\Cref{sep.comp}) which we use for computing IPVT cells as follows. Let $N$ be the corona process. Then the cell of $(\theta,\phi,r)\in \partial \mathbb{H}_{2} \times \partial \mathbb{H}_{2}\times \mathbb{R}_{\geq 0}$ is

$$
C(\theta,\phi,r)\coloneqq \{ z : \bud\bigl(z, (\theta,\phi, r)\bigr) \leq \bud\bigl(z, (\theta',\phi', r')\bigr) \; \text{for all} \; (\theta',\phi',r') \in N \} \;,
$$

with its natural extension at Gromov boundary of $\mathcal{M}$, denoted by $\partial \mathcal{M}$. We then prove that the law of ${\rm IPVT}(\mathcal{M})$ is invariant under \emph{all} the isometries of $\mathcal{M}$ (\Cref{prop.tga}). As in~\cite[Section 5]{IPVT23}, the transitivity of this group action allows us to focus on the cell of the corona point with smallest radius to get the distributional properties of \emph{every} cell of ${\rm IPVT}(\mathcal{M})$. By this method and work of Biermé--Estrade on Poisson random balls~\cite{BEcovering}, we obtain the following result on the topology of the cells of ${\rm IPVT}(\mathcal{M})$ at $\partial \mathcal{M}$. Define the \dfn{end} of the cell $C(\theta,\phi,r)$ as the following union of circles: $\mathcal{E}(\theta,\phi)=\{(\theta,\tau_{2}) \, \ST \, \tau_{2} \in \partial\mathbb{H}_{2} \} \cup \{ (\tau_{1},\phi)\, \ST \,\tau_{1} \in \partial \mathbb{H}_{2} \}$. Then:

\begin{thm}[\textsc{Ends of cells}]\label{thm.1}
Almost surely, the points of each cell of \emph{IPVT}($\mathcal{M}$) at $\partial \mathcal{M}$ lie in its end.
\end{thm}
}
\medskip
\noindent
{
Finally, motivated in part by work of Fr\c{a}czyk--Mellick--Wilkens~\cite{FMW}, we study the separation to any corona point seen from a point $(X_{t},Y_{t})\in \mathcal{M}$ converging (in the Gromov sense) when $t\to \infty$ to either of the two intersections of the ends of two cells of {IPVT}($\mathcal{M}$).  We get the following results:
}
\begin{thm}[\textsc{A.s.~unbounded set at equal separation and limit separation process}]\label{thm.2}$_{}$\\
\begin{itemize}
\vspace{-10pt}
\item[(i)] Almost surely, the set of points $z \in \mathcal{M}$ at equal separation from any corona points is unbounded;
\item[(ii)]  The separation seen from $(X_{t},Y_{t})$ can be rescaled so to converge in distribution when $t\to \infty$ to a stationary Poisson process $S_{\infty}(\eta,\xi)$ over $\mathbb{R}\times \mathbb{R}\times \mathbb{R}_{\geq 0}$ of intensity measure
$$
\nu_{\infty}(\eta,\xi) = 4\pi^{2}\frac{\Re(\eta)}{|\eta+i \hat{\theta}|^{2}} \II{\Re(\eta)\geq 0} \;\mathrm{d}\hat{\theta} \frac{\Re(\xi)}{|\xi+i \hat{\phi}|^{2}}\II{\Re(\xi)\geq 0} \;\mathrm{d}\hat{\phi}\,  \II{y\geq 0}\mathrm{d}y  \,,
$$
where $(\eta,\xi)\in \mathbb{C}\times \mathbb{C}$ and $\Re(\eta)$ denotes the real part of $z \in \mathbb{C}$.
\end{itemize}
\end{thm}

\begin{remark}
Item (i) in \Cref{thm.2} follows by an application of Jeulin's Lemma~\cite[Proposition 4]{jeulin1982} in the neighborhood of $(X_{t},Y_{t})$. This result is the analogous of~\cite[Theorem 6.1]{FMW}, which holds for the rank $2$ symmetric space $\mathcal{X}=(\mathbb{H}_{2}\times \mathbb{H}_{2},L^{2})$ endowed with its natural Riemannian structure. Albeit $\mathcal{M}$ and $\mathcal{X}$ have the same isometry group (see~\Cref{lem.isomp})  and homeomorphic coronas (compare~\Cref{prop.convipvt} and~\cite[Section 4]{FMW}), it is currently unclear whether further (topological, geometrical, distributional) connections between the cells of ${\rm IPVT}(\mathcal{M})$ and ${\rm IPVT}(\mathcal{X})$ exist. I thank the authors of~\cite{FMW} for explaining their work to me. 
\end{remark}

\begin{figure}[!hbtp]
\centering
\begin{minipage}{.45\textwidth}
\begin{overpic}[width=\textwidth]{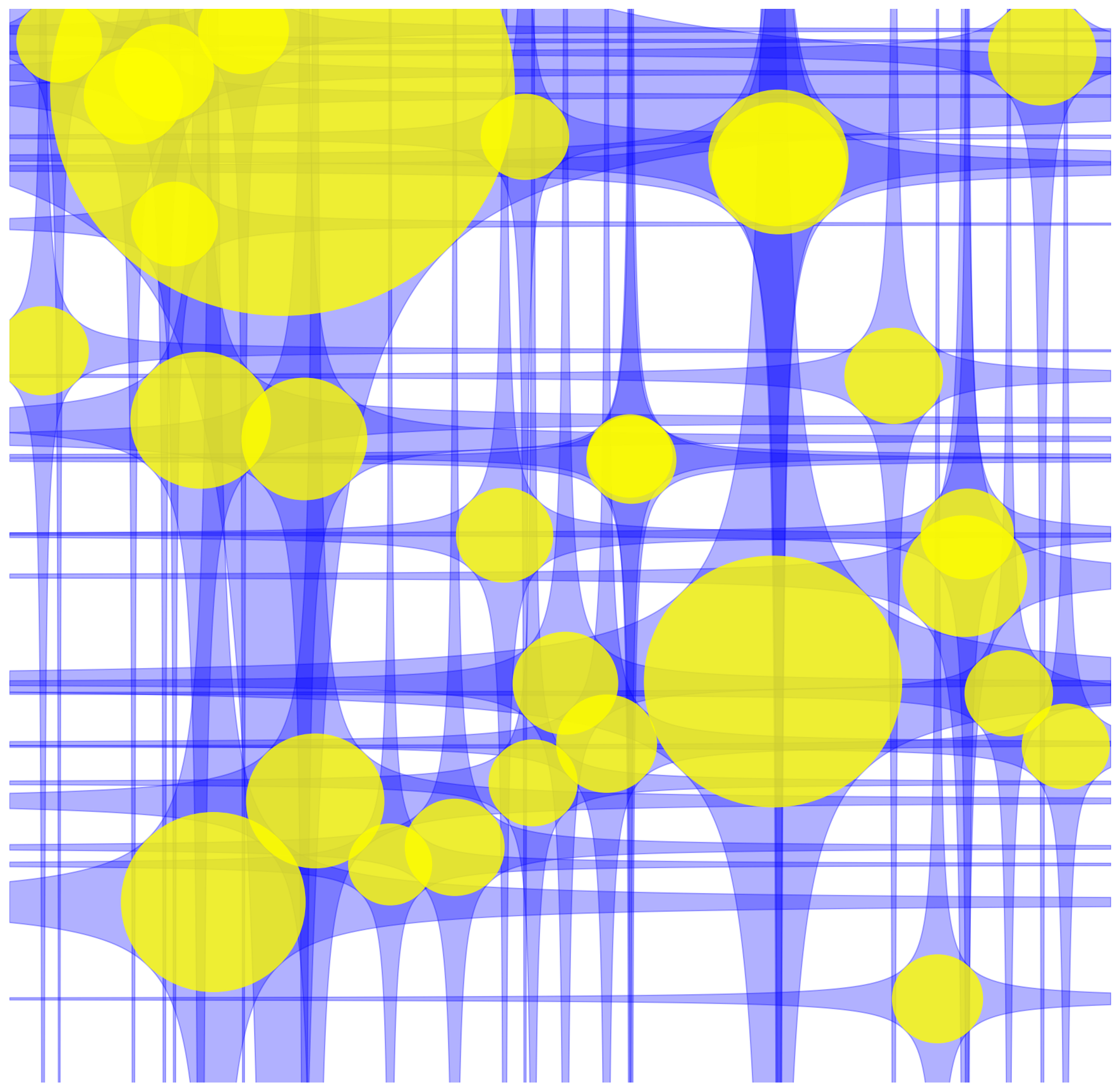}

\end{overpic}   
     \end{minipage}
     \hspace{2pt}
     \begin{minipage}{.45\textwidth}
\begin{overpic}[width=\textwidth]{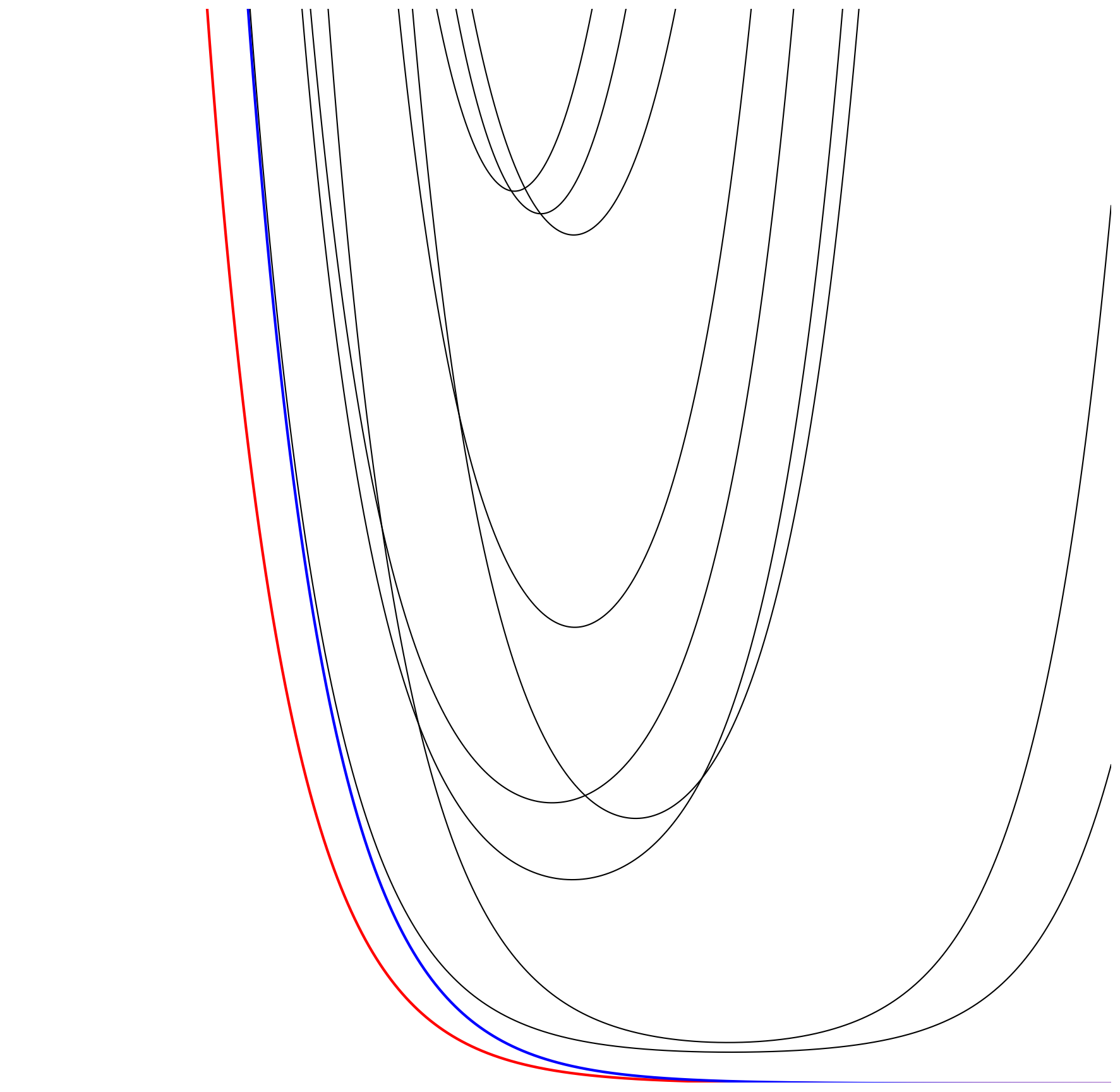}
 \put (50,-2) {{\small{$x$}}}
\end{overpic}   
     \end{minipage}
    \caption{Left: the first 30 points of the deposition model of hyperbolic crosses appearing in the proof of \Cref{thm.1}. Right: the first 10 points of the separation field $S_{\infty}(e^{-x},e^{-x})$ appearing in \Cref{thm.2} plotted against $x$. The red and blue lines represent respectively the separation from the first and second corona points.}
    \label{fig.rain}
\end{figure}

 {Lastly, we report the following easy corollary of an argument used in the proof of \Cref{thm.2}. }
Let $(\Theta_{1},\Phi_{1},R_{1})$ and $(\Theta_{2},\Phi_{2},R_{2})$ be the first and second corona points (when ranked by increasing radii), and let $\mathcal{C}_{2,2}$ be the cell of $(\Theta_{1},\Phi_{1},R_{1})$ in ${\rm IPVT}(\mathcal{M})$. Then:
\begin{cor}[\textsc{Tie-break at infinity}]\label{cor.xuyv}
Let $U$ be a ${\rm Unif}([0,1])$ random variable and let $B_{1},B_{2}$ be two ${\rm Beta}\left(\frac{1}{2},\frac{1}{2}\right)$ random variables, $U,B_{1},B_{2}$ all pairwise independent. Define $Z\overset{\rm law}{=}U\frac{B_{1}}{B_{2}}$. Then
$$
\mathbb{P}((\Theta_{1},\Phi_{2})\in \mathcal{C}_{2,2})=\mathbb{P}(Z\leq 1)=\frac{1}{2}+\frac{2}{\pi^{2}}=0.70264^{+}\; .
$$
\end{cor}

\medskip
\noindent
{\bf Plan of the paper.} \Cref{sec.basics} provides the basics on ${\rm IPVT}(\mathcal{M})$, namely convergence towards ${\rm IPVT}(\mathcal{M})$ in the low-intensity limit (\Cref{prop.convipvt}) and the formulas for computing exponential separation using products of Poisson kernels in two models of $\mathcal{M}$ (\Cref{sep.comp}). \Cref{sec.tga} contains a characterization of the isometry group of $\mathcal{M}$ (\Cref{lem.isomp}). \Cref{sep.comp} and \Cref{lem.isomp} imply that the law of ${\rm IPVT}(\mathcal{M})$ is invariant under the (extended) action of \emph{all} isometries of $\mathcal{M}$ (\Cref{prop.tga}). 
 \Cref{sec.depmod} contains the proof of \Cref{thm.1}; \Cref{thm.2} is proven in \Cref{sec.lsfptt} alongside \Cref{cor.xuyv}.

\medskip

\noindent \textbf{Acknowledgements}. I thank Nicolas Curien for inspiring this work and for his constant support. Many thanks to Bram Petri for permitting the inclusion of his proof of \Cref{lem.isomp}, to Guillaume Blanc and Meltem Ünel for useful conversations  and to Ali Khezeli for a careful reading of the manuscript. I am supported by the ERC Consolidator Grant SuperGRandMA (Grant No.~101087572).

\section{Basics of IPVT($\mathcal{M}$)}\label{sec.basics}

In this section we follow {rather} closely~\cite[Section 3]{IPVT23}, to which the reader is referred. Let $\HH_{2}$ be the hyperbolic plane endowed with hyperbolic distance $\mathrm{d}_{\mathbb{H}_2}$ and let $\partial \HH_{2}$ denote its Gromov boundary (homeomorphic to the unit circle $\mathbb{S}_{1}$ in the Poincaré disk model). Denote by $\mathcal{M}$ the Cartesian product $E=\HH_{2} \times \HH_{2}$ equipped with the $L^{1}$ distance $d: E\times E \to \mathbb{R}_{\geq 0}$ defined, for all $x=(x_{1},x_{2}), y=(y_{1},y_{2}) \in E$, by
\begin{equation}\label{def.dist}
d(x,y)\overset{\rm def}{=}\mathrm{d}_{\mathbb{H}_2}(x_{1},y_{1})+\mathrm{d}_{\mathbb{H}_2}(x_{2},y_{2}) \; ,
\end{equation}
and endow it with the product volume measure $\mathrm{Vol}_{ \mathcal{M}}$.
Denote the origin of $\mathcal{M}$ by $\origin=(\origin_{\HH_{2}},\origin_{\HH_{2}})$. 

\begin{thm}[\textsc{Convergence towards IPVT($\mathcal{M}$), extended version}]\label{prop.convipvt}
Let $\mathbf{X}^{(\lambda)}=(X_i^{(\lambda)} \, \ST \, i\geq 1 )$ be a PPP of intensity measure $\lambda \cdot \mathrm{Vol}_{ \mathcal{M}}$, where the points of $\mathbf{X}^{(\lambda)}$ are ranked by increasing values of $d(X_i^{(\lambda)},\origin)$. Then the following convergence in law holds
$$
\Vor(\mathbf{X}^{(\lambda)}) \underset{\lambda \downarrow 0}{\overset{\rm law}{\Longrightarrow}}\mathrm{Vor}( \boldsymbol{\Theta}, \boldsymbol{\Phi}, \mathbf{D}) \; ,
$$
where $\boldsymbol{\Theta}=(\Theta_{1},\ldots)$ and $\boldsymbol{\Phi}=(\Phi_{1},\ldots)$ are i.i.d.~uniform on $\partial\mathbb{H}_2 \times \partial\mathbb{H}_2$ and $\mathbf{D}=(D_{1},\ldots)$ is s.t.~$\left(\pi^{2}e^{D_{i}}\right)_{i\geq 1}$ is a rate-1 homogeneous PPP on $\mathbb{R}_{\geq 0}$. The processes $\boldsymbol{\Theta}$,  $\boldsymbol{\Phi}$, $\mathbf{D}$ are all pairwise independent. We call $\mathrm{Vor}( \boldsymbol{\Theta}, \boldsymbol{\Phi}, \mathbf{D})$ the \dfn{ideal Poisson–Voronoi tessellation} of $\mathcal{M}$ and denote it by $\text{\rm IPVT}(\mathcal{M})$.

\end{thm}
\rproof
Consider the ball of radius $r$ centered at $\origin  \in \mathcal{M}$
$$
B_{r}(\origin)=\{x \in \mathcal{M} \ST d(x,\origin) \leq r \} = \{x \in \mathcal{M} \ST \exists \rho \in [0,r]: \mathrm{d}_{\mathbb{H}_2}(x_{1},\origin_{\HH_{2}}) \leq \rho  \; \text{\bf and} \;  \mathrm{d}_{\mathbb{H}_2}(x_{2},\origin_{\HH_{2}}) \leq r-\rho \} \; .
$$ 
Working in the product of Poincaré disk models, the $\mathcal{M}$-volume of $B_{r}(\origin)$ is given by
$$
\phi(r)\overset{\rm def}{=} {\rm Vol}_{\mathcal{M}}(B_{r}(\origin)) =\int_{0}^{r}f_{2}(\rho)f_{2}(r-\rho)\,  \mathrm{d}\rho=2\pi^{2}\left(r \cosh{r}-\sinh{r} \right)  \, ,
$$
where $f_{2}$ is the volume function of $\HH_{2}$ (see~\cite[page 12]{IPVT23}). The quantity $\phi(r)$ is of order $r e^{r}$ for large $r$, therefore the first point $X_1^{(\lambda)}$ of the PPP (i.e.~the point nearest to $\origin$) is roughly at a distance $\log{\frac{1}{\lambda}}-\log{\log{\frac{1}{\lambda}}}+o(1)$ from $\origin$ when the intensity $\lambda$ is small (more precisely $\frac{d(X_1^{(\lambda)},\origin)}{|\log{\lambda}|}\underset{\lambda \downarrow 0}{\overset{(\mathbb{P})}{\Longrightarrow}}1$). We can thus readily identify the following notion of \dfn{proto-delays}
\begin{equation}\label{eq.defdelay}
D^{(\lambda)}_{i} \coloneqq d(X_i^{(\lambda)},\origin)-\log{\frac{1}{\lambda}}+\log{\log{\frac{1}{\lambda}}}, \quad i \geq 1 \;  .
\end{equation}
For $0\leq x \leq y$, apply the mapping theorem for Poisson processes to the points of $\mathbf{X}^{(\lambda)}$ at a distance from $\origin$ within the shifted interval $(x+\log{\frac{1}{\lambda}}-\log{\log{\frac{1}{\lambda}}},y+\log{\frac{1}{\lambda}}-\log{\log{\frac{1}{\lambda}}})$. This gives
\begin{equation}\label{eq.delconv}
\left(D^{(\lambda)}_{i}\right)_{i\geq 1} \overset{\rm law}{\Longrightarrow}\left(D_{i}\right)_{i\geq 1} \; ,
\end{equation}
where $\left(D_{i}\right)_{i\geq 1}$ are the increasing points of a PPP on $\mathbb{R}$ with intensity measure $\pi^{2} e^{s} \mathrm{d}s$. Equivalently, the process of radii $\mathbf{R}=\left(R_{i}=\pi^{2}e^{D_{i}}\right)_{i\geq 1}$ is a rate-1 homogeneous PPP on $\mathbb{R}_{\geq 0}$.

\medskip
\noindent
Let $\partial \mathcal{M}=(\mathbb{H}_2 \times \partial\mathbb{H}_2) \cup (\partial\mathbb{H}_2 \times \mathbb{H}_2) \cup (\partial\mathbb{H}_{2} \times \partial\mathbb{H}_{2})$. From the expression for proto-delays in \Cref{eq.defdelay}, it follows that any nuclei converging in the Gromov sense (see~\cite[Section 2.1]{IPVT23}) towards $(\mathbb{H}_2 \times \partial\mathbb{H}_2)$ or $(\partial\mathbb{H}_2 \times \mathbb{H}_2)$ would have a.s.~infinite delay. This gives the weak convergence towards $\partial\mathbb{H}_2 \times \partial\mathbb{H}_2$, which coincides with the visual boundary of $\mathcal{M}$ in the sense of ${\rm CAT(0)}$ spaces. The statement follows from~\cite[Theorem 2.3]{IPVT23}.
\Qed

\bigskip
Following~\cite[Section 3.3]{IPVT23}, we call $ \widetilde{\partial \mathcal{M}}\coloneqq\partial \HH_{2} \times \partial \HH_{2} \times \mathbb{R}_{\geq 0}$ the \dfn{corona}.  
Thus $N\coloneqq(\Theta_{i}, \Phi_{i},\pi^{2}e^{D_{i}})_{i\geq 1}$ is a PPP on the corona of intensity measure
\begin{equation}\label{eq.cormes}
\mu= \Unif \otimes \Unif \otimes \Leb_{\mathbb{R}_{\geq 0}}  \; ,
\end{equation}
see \Cref{f.corona} for a portrait. 
\begin{figure}[!hbtp] 
\centering
\includegraphics[width=.5\textwidth]{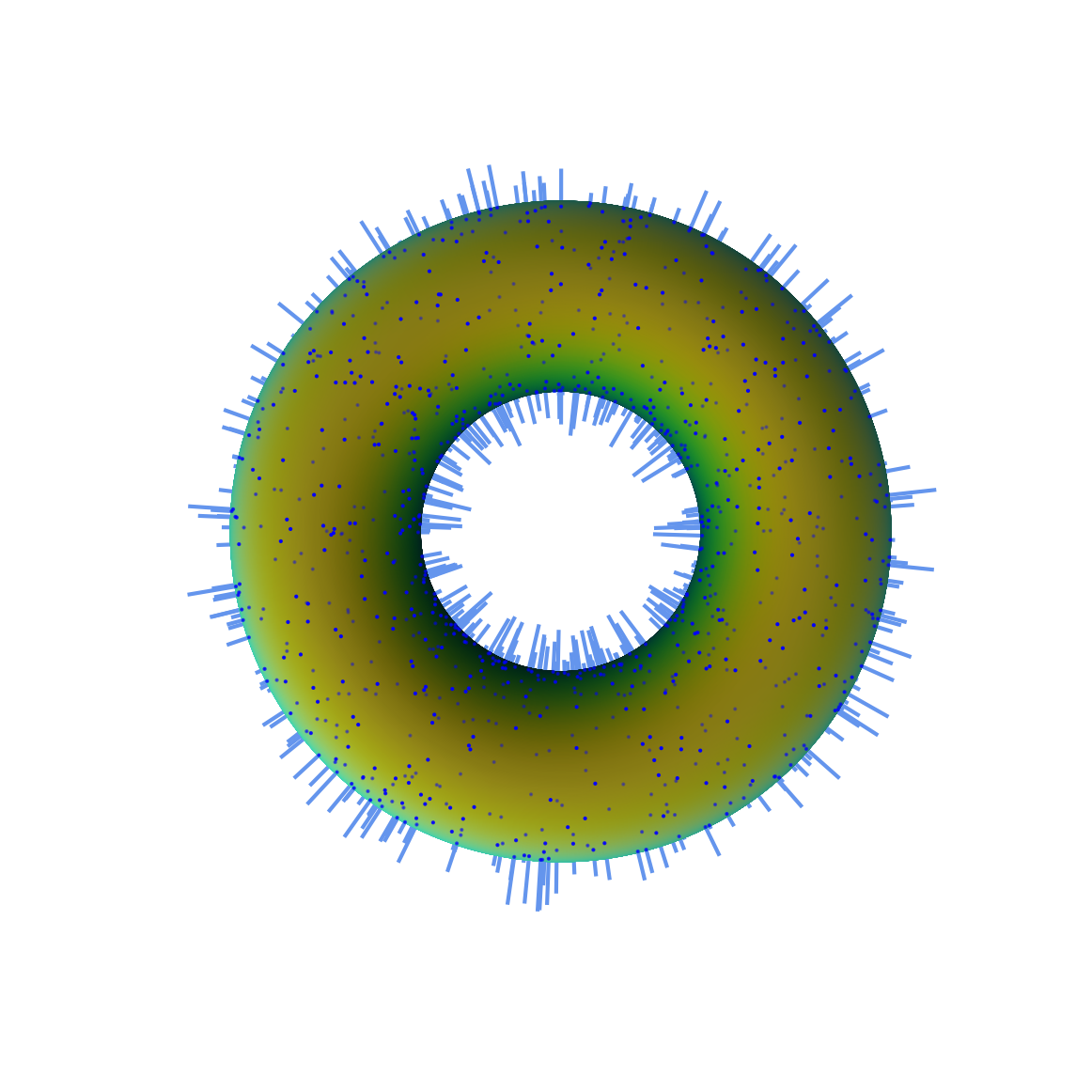}
\caption{Portrait of the corona of $\mathcal{M}$ showing
the first 1000 corona points. The radii of the nuclei are scaled linearly to improve visibility. Each point $(\theta,\phi, r)$ in the corona is joined by a blue segment to its projection $(\theta,\phi)$ onto $\partial \mathbb{H}_{2}\times \partial \mathbb{H}_{2}$, represented here by a yellow/green $2$-torus.
}
		\label{f.corona}
\end{figure}

Working via low--intensity limits, we get the following explicit formulas for the exponential separation:

\begin{prop}[\textsc{Computing the separation}]\label{sep.comp}
The separation from $z=(z_{1},z_{2}) \in \mathcal{M}$ to $(\theta, \phi, r) \in  \widetilde{\partial \mathcal{M}}$ is given by
\[
\bud\bigl(z, (\theta,\phi, r)\bigr)
=
\rcases
{\displaystyle \frac{r}{K(z_{1}, \theta)K(z_{2}, \phi)} &\mbox{in the product of discs model,}\\[10pt]
\displaystyle \frac{r}{\Kh(z_{1}, \theta)\Kh(z_{2}, \phi)} &\mbox{in the product of upper half-planes model,}\\
}
\]
where $K(z, \theta)$ and $\Kh(z, \theta)$ are respectively the hyperbolic Poisson kernel and the modified hyperbolic Poisson kernel defined in~\cite[Section 3.3.1]{IPVT23}.

\end{prop}

\rproof Work in each factor as in \cite[Proof of Lemma 3.3]{IPVT23}.
\Qed

\section{Transitive group action of the isometry group of $\mathcal{M}$ on the corona measure}\label{sec.tga}

In this section we provide a transitive action of the isometry group of $\mathcal{M}=(\mathbb{H}_{2}\times\mathbb{H}_{2}, L^{1})$ on the corona $\widetilde{\partial \mathcal{M}}$ which leaves the measure in \Cref{eq.cormes} invariant. Recall that a map $g: \mathcal{M} \to \mathcal{M}$ is an \dfn{isometry} if, $\forall x,y \in \mathcal{M}$,
$$
\mathrm{d}_{\mathbb{H}_2}\bigl((g(x))_{1},(g(y))_{1})+\mathrm{d}_{\mathbb{H}_2}((g(x))_{2},(g(y))_{2}\bigr)=\mathrm{d}_{\mathbb{H}_2}(x_{1},y_{1})+\mathrm{d}_{\mathbb{H}_2}(x_{2},y_{2}) \; .
$$
Let ${\rm Isom}_{1}$ be the group of all such isometries endowed with composition. Then \emph{any} $g \in {\rm Isom}_{1}$ is a direct product of two isometries of $\mathbb{H}_{2}$, whose group we denote by $\mob_2$ (see \cite[page 11]{IPVT23} for a definition) semi-direct with an involution exchanging the factors. Quite surprisingly, we could not find this result in the literature thus we include a proof here:

{
\begin{lem}[\textsc{Isometry group of $\mathcal{M}$}]\label{lem.isomp}
$$
{\rm Isom}_{1}\cong \left(\mob_2 \times \mob_2 \right)\rtimes \mathbb{Z}/ (2 \mathbb{Z})\; ,
$$
where $\cong$ denotes isomorphism, $ \times$ denotes direct product and $\rtimes$ denotes the semi-direct product.
\end{lem}
\begin{proof}
One can prove this result by showing that ${\rm Isom}_{1}\cong {\rm Isom}_{2}$, where ${\rm Isom}_{2}$ is the group of isometries of $(\mathbb{H}_{2}\times \mathbb{H}_{2},L^{2})$, and the latter is known to be $\left(\mob_2 \times \mob_2 \right)\rtimes \mathbb{Z}/ (2 \mathbb{Z})$ (see~\cite[Proposition 2.2]{charette2017}). However, since \Cref{lem.isomp} is possibly surprising when compared to its Euclidean counterpart, we include here a constructive proof.

\medskip
One direction of the proof is obvious, we only need to prove that any $g\in {\rm Isom}_{1}$ lies in $(\mob_2\times \mob_2) \rtimes\mathbb{Z}/ (2 \mathbb{Z})$. Let us work in the model of $\mathcal{M}$ given by the product of upper-half planes $\mathbb{U}_{2} \times \mathbb{U}_{2}$, in which $\origin=(i,i)$. Here, the group action of $\mob_2 \times \mob_2$ is given by $SL_{2}(\mathbb{R})\times SL_{2}(\mathbb{R})$ (we include orientation-reversing isometries). Since $SL_{2}(\mathbb{R})\times SL_{2}(\mathbb{R})$ acts transitively on $\mathcal{M}$, there exists $h \in {\rm Isom}_{1}$ s.t.~$h\left(g (i,i)\right)=(i,i)$. In particular, $g \in (SL_{2}(\mathbb{R})\times SL_{2}(\mathbb{R}))({\rm Stab}((i,i)))$, where ${\rm Stab}((i,i))$ is the stabilizer of the point $(i,i)$ in the isometry group of $\mathcal{M}$. In order to conclude, we need to show that ${\rm Stab}((i,i)) \cong (O(2) \times O(2)) \rtimes \mathbb{Z}/ (2 \mathbb{Z})$, where $O(2)$ is the orthogonal group of dimension $2$.

\medskip
First, any element $t \in {\rm Stab}((i,i))$ acts on the tangent space at $(i,i)$, which we denote by $T_{(i,i)}$. $T_{(i,i)}$ comes with a norm which induces the distance on $\mathcal{M}$ (in the sense of Finsler manifolds). The latter norm is an $L^1$ combination of the two $L^{2}$ norms of the two factors, and it is preserved by any such element $t$. Therefore, the derivative of any element in ${\rm Stab}((i,i))$ lies in $(O(2) \times O(2)) \rtimes \mathbb{Z}/ (2 \mathbb{Z})$, and the latter is the linear isometry group of this norm on the tangent space due to the Mazur--Ulam theorem. In order to conclude, use the exponential map to get back to $\mathcal{M}$ to show that there exists a unique such isometry $g$.
\end{proof}
}

\medskip
\noindent
Combining \Cref{eq.cormes}, \Cref{sep.comp} and \Cref{lem.isomp} gives:
\begin{cor}[\textsc{Transitive group action of ${\rm Isom}_{1}$}]\label{prop.tga}

For any isometry $g=(g_{1},g_{2})\in {\rm Isom}_{1}$, the action on a corona point $(\theta,\phi,r)\in \widetilde{\partial \mathcal{M}}$ defined in the product of disks model
$$
g(\theta,\phi,r) \coloneqq \left(g_{1}(\theta), g_{2}(\phi),\frac{r}{K(g_{1}^{-1}\left(\origin),\theta\right) \cdot K(g_{2}^{-1}\left(\origin),\phi\right)}\right)
$$
is a transitive group action that leaves the corona measure $\mu$ in \eqref{eq.cormes} invariant. Consequently, the law of ${\rm IPVT}(\mathcal{M})$ is invariant under ${\rm Isom}_{1}$.
\end{cor}

\noindent

\section{Ends of cells via hyperbolic crosses}\label{sec.depmod}

\medskip
In this section we work in the model of $\mathcal{M}$ given by the product of upper-half planes $\mathbb{U}_{2} \times \mathbb{U}_{2}$ and focus on the cell of the corona point $(\theta_{1},\phi_{1},r_{1})$ having the smallest radius. We call the cell $C(\theta_{1},\phi_{1},r_{1})$ the \dfn{zero cell} and denote it by $\mathcal{C}_{2,2}$. Contrary to the zero cell of $\text{IPVT}(\mathbb{H}_{d})$\footnote{See~\cite[Fig.~1.2]{IPVT23} for $d=2,3$ and \href{https://skfb.ly/oDVU8}{https://skfb.ly/oDVU8} for an interactive model in $d=3$.}, the description of $\mathcal{C}_{2,2}$ is less straightforward. However by \Cref{prop.tga} we can study the distributional properties of $\mathcal{C}_{2,2}$ to get those of the cell of ${\rm IPVT}(\mathcal{M})$ containing an arbitrary point. Given two corona points $(\theta_{1},\phi_{1},r_{1})$ and $(\theta_{2},\phi_{2},r_{2})$, introduce their \dfn{no man's land}, denoted by $\text{NML}((\theta_{1},\phi_{1}, r_{1}),(\theta_{2},\phi_{2}, r_{2}))$, as the locus of points $z\in \mathcal{M}$ s.t.~$\bud\bigl(z, (\theta_{1},\phi_{1}, r_{1})\bigr)=\bud\bigl(z, (\theta_{2},\phi_{2}, r_{2})\bigr)$ with the obvious extension at $\partial \mathcal{M}$. Recall the modified kernel $\hat{K}$ in \cite[Page 15]{IPVT23} and \Cref{sep.comp}. Then:

\begin{cor}[\textsc{Geometry of the no man's land in the product of upper-half planes}]\label{cor.nmlinfinf}
Let $(\theta_{1},\phi_{1},r_{1})$ and $(\theta,\phi,r)$ be two corona points.
In the model of $\mathcal{M}$ given by the product of upper-half planes $\mathbb{U}_{2} \times \mathbb{U}_{2}$:
\begin{itemize}
\item[(i)] If $(\theta_{1},\phi_{1})$ is sent at $(\infty,\infty)$,

$$
\emph{\text{NML}}((\infty,\infty, r_{1}),(\theta,\phi, r)) = \left \lbrace (z_{1},z_{2})\in \mathcal{M}\ST  |z_{1}-\theta|^{2} \cdot |z_{2}-\phi|^{2}=\frac{r_{1}}{r}\left(1+|\theta|^{2} \right)\left(1+|\phi|^{2} \right)\right \rbrace \; ;
$$
\item[(ii)] If $(\theta_{1},\phi_{1})$ is sent at $(\infty,y)$, for $y\in \mathbb{R}$,
{
$$
\emph{\text{NML}}((\infty,y, r_{1}),(\theta,\phi, r)) = \left \lbrace (z_{1},z_{2})\in \mathcal{M} \ST  \frac{|z_{1}-\theta|^{2} \cdot |z_{2}-\phi|^{2}}{|z_{2}-y|^{2}}=\frac{r_{1}}{r}\frac{\left(1+|\theta|^{2} \right)\left(1+|\phi|^{2} \right)}{1+|y|^{2}}\right \rbrace \;.
$$
}

\end{itemize}
\end{cor}

\medskip
We will now prove that, almost surely, all points of $\mathcal{C}_{2,2}$ at $\partial \mathbb{H}_{2} \times \partial \mathbb{H}_{2}$, are included in the set $\mathcal{E}(\Theta_{1},\Phi_{1})=\{(\Theta_{1},\tau_{2}) \, \ST \, \tau_{2} \in \partial\mathbb{H}_{2} \} \cup \{ (\tau_{1},\Phi_{1})\, \ST \,\tau_{1} \in \mathbb{H}_{2} \}$, which is the end of $\mathcal{C}_{2,2}$. This proves \Cref{thm.1} by \Cref{prop.tga}. An important ingredient of the proof is the following elementary result from Euclidean geometry, whose proof is left as an exercise:

\begin{lem}[\textsc{Largest disk inscribed in a hyperbolic cross}]\label{lem.kl}
For $(a,b)\in \mathbb{R}_{2}$ and $c\in \mathbb{R}_{\geq 0}$, define the \dfn{hyperbolic cross} ${\rm HC}(a,b,c)\coloneqq\{ (x,y) \in \mathbb{R}_{2} \sut (x-a)^{2}(y-b)^{2}\leq c (1+a^{2})(1+b^{2})\}$.
\medskip
Then the largest disk contained in ${\rm HC}(a,b,c)$ has center $C=(a,b)$ and radius $\rho=\sqrt{2}\left[c \,  (1+a^{2})(1+b^{2})) \right]^{1/4}$.

\end{lem}
\medskip
\begin{proofof}{\Cref{thm.1}}
The proof is inspired by~\cite[Proof of Theorem 1.4]{IPVT23}. Consider the corona process whose intensity measure is given in \Cref{eq.cormes}, and let $(\Theta_{1},\Phi_{1},R_{1})$ be the corona point with the smallest radius. For $i\geq 2$, let $(\Theta_{i},\Phi_{i},R_{i})$ be any other corona point. In the model of $\mathcal{M}$ given by the product of upper-half planes $\mathbb{U}_{2} \times \mathbb{U}_{2}$, in which $(\Theta_{1},\Phi_{1})$ is sent at $(\infty,\infty)$, define the following plane
\begin{equation}\label{eq.btpl}
 \mathcal{P}\coloneqq\{ (\rm{Re}(z_{1}),\rm{Re}(z_{2})) \sut \rm{Im}(z_{1})=\rm{Im}(z_{2})=0, (z_{1},z_{2}) \in \mathcal{M} \} \; .
 \end{equation}
Then the set of points in $\mathcal{P}$ at smaller separation to $(\Theta_{i},\Phi_{i},R_{i})$ than to $(\Theta_{1},\Phi_{1},R_{1})$ is the hyperbolic cross ${\rm HC}\left(\mathrm{Ste}({\Theta}_{i}),\mathrm{Ste}({\Phi}_{i}),\frac{R_{1}}{R_{i}}\right)$, where $\mathrm{Ste}({\Phi}_{i}) \in \mathbb{R}$ denotes the stereographic projection of $\Phi_{i}$ (analogously for $\Theta_{i}$). Equivalently,
$$
\partial \left({\rm HC}(\mathrm{Ste}({\Theta}_{i}),\mathrm{Ste}({\Phi}_{i}),\frac{R_{1}}{R_{i}})\right) = \mathcal{P} \cap{\rm NML}((\infty,\infty,R_{1}),({\Theta}_{i},{\Phi}_{i},R_i)),
$$ 
where $\partial\left(A\right)$ denotes the boundary of the measurable set $A$ in the topology of $\mathbb{R}_{2}$.
By~\cite[Lemma 3.1]{IPVT23}, the PPP 
$\left(\mathrm{Ste}({\Theta}_{i}),\mathrm{Ste}({\Phi}_{i}),R_i-R_1\, \ST \, i\geq 2\right)$ has the following intensity measure 
\begin{equation}\label{eq.depmod}
\frac{1}{\pi}\frac{1}{\bigl(1+x^{2}\bigr)} \,\mathrm{d}x \otimes \frac{1}{\pi}\frac{1}{\bigl(1+y^{2}\bigr)} \,\mathrm{d}y \otimes   \mathrm{d}t \,\mathbf{1}_{t >0} 
\end{equation}

\noindent
in $\mathbb{R}_{2}\times \mathbb{R}_{+}$. This provides a deposition model of hyperbolic crosses (the unbounded blue regions in \Cref{fig.rain}, left). By \Cref{lem.kl}, ${\rm HC}\left(\mathrm{Ste}({\Theta}_{i}),\mathrm{Ste}({\Phi}_{i}),\frac{R_{1}}{R_{i}}\right)$ contains the disk of center $\left(\mathrm{Ste}({\Theta}_{i}),\mathrm{Ste}({\Phi}_{i})\right)$ and radius $\sqrt{2}\left[\frac{R_{1}}{R_{i}} \,  (1+|\mathrm{Ste}({\Theta}_{i})|^{2})(1+|\mathrm{Ste}({\Phi}_{i})|^{2})) \right]^{1/4}$. For $i \geq 2$, these disks provide a Poisson random ball model (yellow disks in~\Cref{fig.rain}, left). We now prove that the latter covers $\mathcal{P}$ a.s.
Conditionally on $R_{1}=s$, apply the Poisson mapping theorem wrt the change of variables $\rho = \sqrt{2}\bigl(\frac{s}{s+t}(1+|x|^{2})(1+|y|^{2}) \bigr)^{\frac{1}{4}}$ and get, for any test function $f$,
\begin{equation}\label{condintensm}
\begin{split}
   \frac{1}{\pi^{2}} \int_{\mathcal{P}\times \mathbb{R}_{+}} f\Bigl(x, &\sqrt{2}\Bigl((1+|x|^{2})(1+|y|^{2})\left(\frac{s}{s+t}\right)\Bigr)^{\frac{1}{4}}\Bigr) \frac{1}{\left(1+|x|^{2} \right)}\, \frac{1}{\left(1+|y|^{2} \right)}\,  \mathrm{d}x  \mathrm{d}y\, \mathrm{d}t\\
& =  \frac{16 s}{\pi^{2}} \int_{\mathcal{P}\times \mathbb{R}_{+}} f\left(x,y,\rho \right) \frac{1}{\rho^{5}} \mathbf{1}_{\rho^{4}\leq (1+|x|^{2})(1+|y|^{2})} \,\mathrm{d}x \mathrm{d}y \,  \mathrm{d}\rho \; .
\end{split}
\end{equation}
The above (conditional) intensity measure coincides with a Poisson random ball model on $\mathcal{P}\times \mathbb{R}_{+}$ considered by Biermé--Estrade~\cite{BEcovering}\footnote{In the notation of~\cite[Section~4.1]{BEcovering}, $\beta =4 > 2= d={\rm dim}(\mathcal{P})$, where our $\mathcal{P}$ is defined in \eqref{eq.btpl}.} in which ``large'' disks s.t.~$\rho^{4}\geq (1+|x|^{2})(1+|y|^{2})$ are excluded. However, the Biermé--Estrade process covers a.s.~$\mathcal{P}$ due to ``many small balls'' s.t.~$\rho^{2}\leq 1$ (high frequency covering), and so does our model, and hence the hyperbolic crosses by comparison.
\clearpage

Now, for $y \in \mathbb{R}$, perform a Cayley transform (see~\cite[Page 11]{IPVT23}) which sends $(\Theta_{1},\Phi_{1})$ to $(\infty,y)$. Then the region of $\mathcal{P}$ at smaller separation to $(\Theta_{i},\Phi_{i},R_{i})$ than to $(\Theta_{1},\Phi_{1},R_{1})$ is, \underline{almost surely}, an unbounded, ``mushroom-like'' region meeting the axis ${\rm Re}(z_{2})=y$ at the single point $\mathrm{Ste}({\Theta}_{i})$, see \Cref{f.chap}. Hence, the line ${\rm Re}(z_{2})=y$ is not covered a.s.~. An analogous result holds by sending $(\Theta_{1},\Phi_{1})$ via a Cayley transform to $(x,\infty)$, for any $x \in \mathbb{R}$ (now the a.s.~uncovered region would be a vertical straight line through $\mathrm{Ste}(\Phi_{i})$). Finally, pass to the unconditional version by recalling that $R_{1}$ is an ${\rm Exp}(1)$ random variable, and the statement follows.

\begin{figure}[!hbtp] 
\centering
\begin{overpic}[width=.45\textwidth]{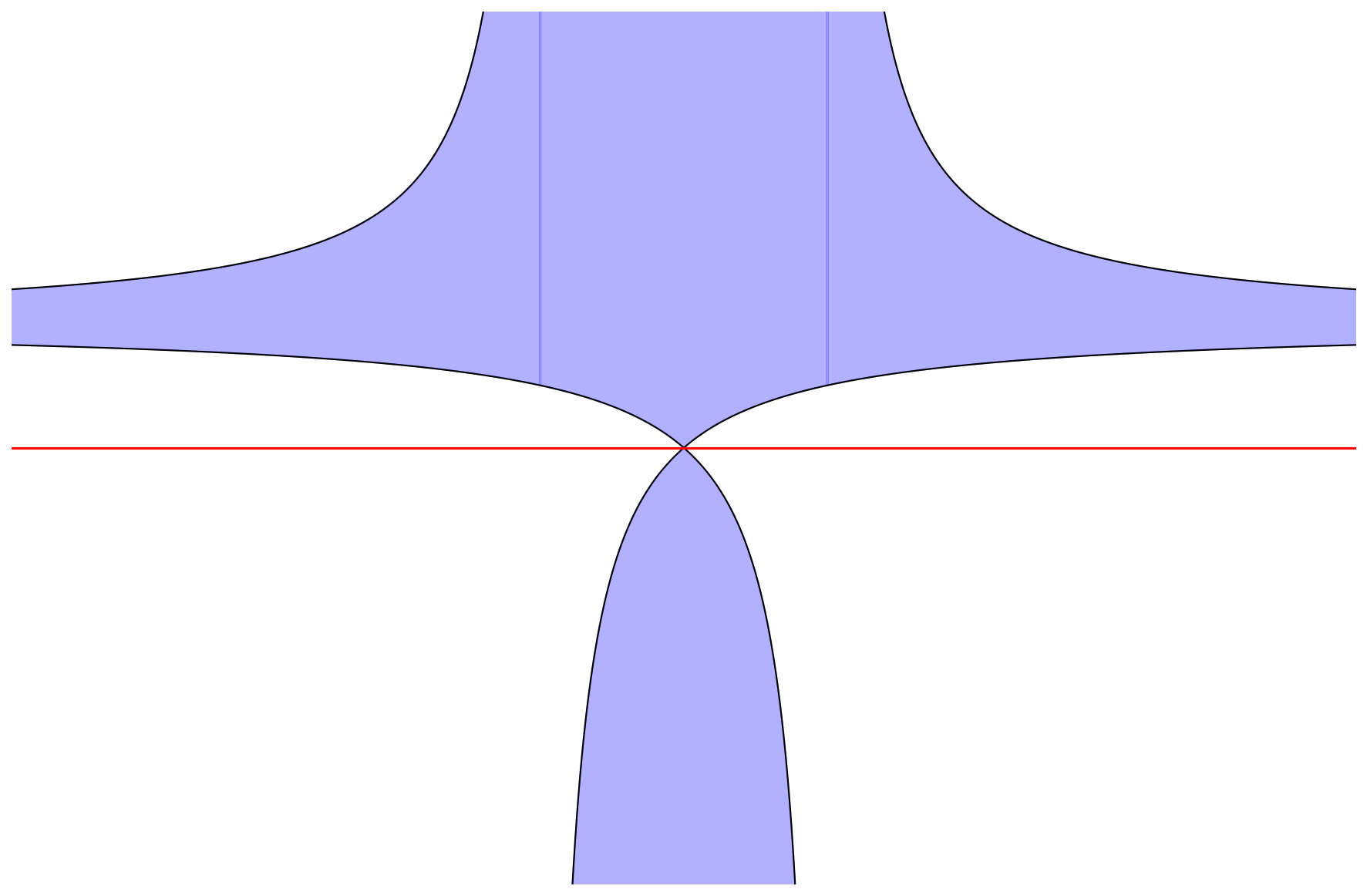}
 \put (-25,32) {{\small{\color{red}${\rm Re}(z_{2})=y$}}}
  \put (49.25,32.18) {{\tiny$\bullet$}}
    \put (44.25,38.18) {\small $\mathrm{Ste}({\Theta}_{i})$}
\end{overpic}
\caption{In blue, the region at smaller separation to $(\Theta_{i},\Phi_{i},R_{i})$ than to $(\infty,y,R_{1})$. The boundary of this region, represented here in black, is $\mathcal{P}\cap {\rm NML}((\infty,y,R_{1}),(\Theta_{i},\Phi_{i},R_{i}))$ (see \Cref{cor.nmlinfinf}).}
		\label{f.chap}
\end{figure}

\end{proofof}

\section{The landscape seen from a point traveling towards $(\theta_{1},\phi_{2})$}
\label{sec.lsfptt}

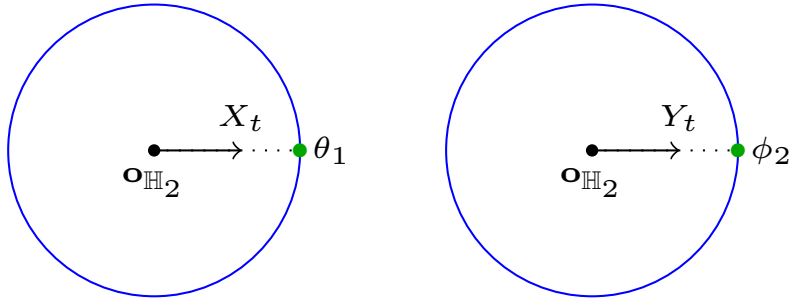
\begin{figure}[!hbtp] 
\centering
\resizebox{.66\linewidth}{!}{\begin{tikzpicture}

\def\t{0.6} 

\draw[blue] (-1.5,0) circle (1); 
\filldraw[black] (-1.5,0) circle (1pt) node[above right] {}; 
\draw[dotted] (-1.5,0) -- (-0.5,0); 
\draw[black, ->] (-1.5, 0) -> (-1.5 + \t, 0) node[above] {\tiny $X_t$}; 
\draw[green!65!black] (-0.5,0) node {\tiny $\bullet$};
\node[right] at (-0.55, 0) {\tiny$\theta_1$};
\draw[black] (-1.5,-0.22) node{\tiny $\origin_{\mathbb{H}_2}$};

\draw[blue] (1.5,0) circle (1); 
\filldraw[black] (1.5,0) circle (1pt) node[above left] {}; 
\draw[dotted] (1.5,0) -- (2.5,0); 
\draw[black, ->] (1.5, 0) -> (1.5 + \t, 0) node[above] {\tiny $Y_t$}; 
\draw[green!65!black] (2.5,0) node {\tiny $\bullet$};
\draw[black] (1.5,-0.22) node{\tiny $\origin_{\mathbb{H}_2}$};
\node[right] at (2.45, 0) {\tiny $\phi_2$};
\end{tikzpicture}
}
\caption{The ``double quarter-past-three'' reference system in $\mathbb{B}_{2}\times \mathbb{B}_{2}$ used in this section.
}
		\label{f.refsyst}
\end{figure}

In this section we work in the model of $\mathcal{M}$ given by the product of Poincaré disks $\mathbb{B}_{2}\times \mathbb{B}_{2}$  conditionally on $(\Theta_{1},\Phi_{1})=(0,\Phi_{1})$ and $(\Theta_{2},\Phi_{2})=(\Theta_{2},0)$ (see \Cref{f.refsyst}).

\medskip
\begin{proofof}{\Cref{thm.2}}For $t\geq 0$, let $\rho_{\rm E}(t) \coloneqq\tanh{\frac{t}{2}}$ and consider the point $(X_{t},Y_{t})=\rho_{\rm E}(t)\cdot(1,1)$ traveling towards the point $(1,1) \in \partial\mathbb{B}_{2}\times \partial \mathbb{B}_{2}$ as $t \to \infty$. We first derive a formula for the separation of $(X_{t},Y_{t})$ to any corona point $(\theta,\phi,r)$. For any $\chi=\rho e^{i \tau}\in \mathbb{B}_{2}$ and $\psi \in \partial \mathbb{B}_{2}$, use the following expression for the hyperbolic Poisson kernel $K(\chi,\psi)=\frac{1-\rho^{2}}{1+\rho^{2}-2 \rho\cos{(\tau -\psi)}}$ in each factor of \Cref{sep.comp}. This gives
\begin{equation}
\label{eq.fol}
\bud\bigl((X_{t},Y_{t}), (\theta,\phi, r)\bigr) = r \cdot \left[f_1(\eps)+2\cdot f_2(\eps) (2-\cos\theta-\cos{\phi})+4\cdot f_3(\eps) (1-\cos\theta)(1-\cos{\phi}) \right]
\end{equation}
\noindent
where  $\varepsilon \coloneqq 1 - \rho_{\rm E}(t)$ and $f_1(\eps)  \coloneqq \left(\frac{\eps}{2-\eps}\right)^{2}$, $f_2(\eps)  \coloneqq \frac{1-\eps}{(2-\eps)^{2}}$ and $f_3(\eps)  \coloneqq\frac{1-\eps}{2\eps^{2}(2-\eps)^{2}}$ are three strictly positive functions over $[0,1)$. It follows that $\forall \eps >0$, there exists $\epsilon>0$ and an $\mathcal{M}$-ball $B_{\delta}(X_{t},Y_{t})$ centered at $(X_{t},Y_{t})$ of radius $\delta>0$ s.t.~the following event 
$$
A_{\eps,\delta}=\{B_{\delta}(X_{t},Y_{t}) \cap \text{NML}((0,\Phi_{1}, R_{1}),(\Theta_{2},0, R_{2})) \; \text{is not empty}\}
$$
holds with probability greater then $\epsilon$. It then follows by Jeulin's Lemma~\cite[Proposition 4]{jeulin1982}\footnote{See~\cite[Proof of Theorem 10.13]{NicoStFlour} for a recent application (and proof) of Jeulin's in the context of random planar maps.} that the set $\text{NML}((0,\Phi_{1}, R_{1}),(\Theta_{2},0, R_{2}))$ is unbounded a.s.~The first item of \Cref{thm.2} follows by passing to the unconditional version and by \Cref{prop.tga}.

\medskip
Before proving the second item of \Cref{thm.2} we derive a preliminary result to build intuition. First, observe that \Cref{eq.fol} provides a foliation of  $\widetilde{\partial \mathcal{M}}$ since $f_{1},f_{2}$ and $f_{3}$ are non-negative functions. Second, $\forall \eps >0$ introduce the rescaled corona process $(\boldsymbol{\hat{\Theta}},\boldsymbol{\hat{\Phi}},\hat{\mathbf{R}})\overset{\rm law} {=}(\eps \boldsymbol{\hat{\Theta}},\eps\boldsymbol{\hat{\Phi}}, \eps^{-2}\hat{\mathbf{R}})\bigr)$. Then
$$
\tilde{\mathcal{S}}^{{(\eps)}}_{t}(\origin)\coloneqq\bud\bigl((X_{t},Y_{t}), (\eps \boldsymbol{\hat{\Theta}},\eps\boldsymbol{\hat{\Phi}}, \eps^{-2}\hat{\mathbf{R}})\bigr)
$$
is a PPP of intensity measure $\tilde{\nu}_{t}$ given by the pushforward of the corona measure $\mu$ (see \Cref{eq.cormes}) via the measurable map in \Cref{eq.fol}. Expanding the cosines in~\Cref{eq.fol} for $\eps$ small (i.e.~$t\to \infty$), $\forall y >0$ the following uniform convergence holds
$$
\tilde{\nu}_{t}([0,y]) \coloneqq \mu([0,\bud^{-1}\{y\}])= 4 y \int_{[-\pi/\eps,\pi/\eps]^{2}} \frac{1}{1+\hat{\theta}^{2}+\hat{\phi}^{2}+\hat{\theta}^{2} \hat{\phi}^{2}+O(\eps^{2})}\, \mathrm{d}\hat{\theta} \mathrm{d}\hat{\phi} \underset{\substack{t \to \infty\\ (\eps \downarrow 0)}}{\longrightarrow}4 \pi^{2} \, y \coloneqq \tilde{\nu}_{\infty}([0,y]) .
$$
Proceeding along the same lines, we get $\tilde{\mathcal{S}}^{(\eps)}_{t}(\origin)\underset{t \to \infty}{\overset{\rm law}{\Longrightarrow}}\tilde{\mathcal{S}}_{\infty}(\origin)$, where $\tilde{\mathcal{S}}_{\infty}(\origin)$ is a stationary PPP over $\mathbb{R}\times \mathbb{R} \times  \mathbb{R}_{\geq 0}$ of intensity measure $\tilde{\nu}_{\infty}=4\pi^{2} \cdot \frac{1}{1+\hat{\theta}^{2}} \;\mathrm{d}\hat{\theta} \; \frac{1}{1+\hat{\phi}^{2}} \,\mathrm{d}\hat{\phi} \,\II{y\geq 0} \mathrm{d}y$ (see \Cref{fig.rain} for a portrait). We now repeat the above argument and get for suitable $z=(z_{1},z_{2}) \in \mathcal{M}$ the convergence in law of the following separation process seen from $z$
$$
\tilde{\mathcal{S}}^{{(\eps)}}_{t}(z)\coloneqq \bud\bigl(z, (\eps \boldsymbol{\hat{\Theta}},\eps\boldsymbol{\hat{\Phi}}, \eps^{-2}\hat{\mathbf{R}})\bigr) \; .
$$
Let $g=(\varphi,\varphi)\in {\rm Isom}_{1}$ be the unique isometry fixing $(1,1)\in \partial\mathbb{B}_{2}\times \partial \mathbb{B}_{2}$ s.t.~$(\varphi(X_{t}),\varphi(Y_{t}))=\origin$. By \Cref{prop.tga},

$$
\tilde{\mathcal{S}}^{{(\eps)}}_{t}(z) \overset{\rm law}{=}\tilde{\mathcal{S}}^{{(\eps)}}_{t}(\origin) \cdot \frac{1}{K(\varphi(z_{1}),\varphi(\eps \boldsymbol{\hat{\Theta}}))}\cdot \frac{1}{K(\varphi(z_{2}),\varphi(\eps \boldsymbol{\hat{\Phi}}))} \; .
$$

\smallskip
\noindent
Use that $\forall \varphi \in \mob_{2}, \forall z \in \mathbb{B}_{2}$ and $\theta \in \partial \mathbb{B}_{2}$ (see \emph{e.g.}~\cite[page 8]{beardon2012geometry}) the hyperbolic Poisson kernel satisfies
$$
K(\varphi(z),\varphi(\theta))=\frac{K(z,\theta)}{|\varphi'(\theta)|},
$$
and put $z=(1-\eps \eta, 1-\eps \xi)$ for $(\eta,\xi)\in \mathcal{M}$ s.t.~$\Re(\eta)\geq \frac{1}{2}\eps |\eta|^{2}$ and $\Re(\xi)\geq \frac{1}{2}\eps |\xi|^{2}$. After some tedious but elementary calculation
$$
\tilde{\mathcal{S}}^{{(\eps)}}_{t}(1-\eps \eta, 1-\eps \xi) \underset{t \to \infty}{\overset{\rm law}{\Longrightarrow}} S_{\infty}(\eta,\xi)
$$
where the intensity measure of $S_{\infty}(\eta,\xi)$ is given in the statement.
\end{proofof}
\medskip
\begin{proofof}{\Cref{cor.xuyv}}
Plug either $(\theta,\phi)=(\Theta_{1},\Phi_{1})$ or $(\theta,\phi)=(\Theta_{2},\Phi_{2})$ in \Cref{eq.fol}. Then the coefficient of $f_{2}$ is $1-\cos{(\Phi_{1})}$, respectively $1-\cos{(\Theta_{2})}$, and in both cases the coefficient of $f_{3}$ vanishes identically. Thus the corresponding separations converge in law to $\bud\bigl((\Theta_{1},\Phi_{2}),(\Theta_{1},\Phi_{1},R_{1})\bigr)$, respectively $\bud\bigl((\Theta_{1},\Phi_{2}),(\Theta_{2},\Phi_{2},R_{2})\bigr)$, when {$t \to \infty$}. The event $(\Theta_{1},\Phi_{2})\in \mathcal{C}_{2,2}$ corresponds to the event that $\frac{\bud\bigl((\Theta_{1},\Phi_{2}),(\Theta_{1},\Phi_{1},R_{1})\bigr)}{\bud\bigl((\Theta_{1},\Phi_{2}),(\Theta_{2},\Phi_{2},R_{2})\bigr)}\leq 1$. The rest is straightforward computation using the two following well-known facts:
\begin{itemize}
\item If $X$, $Y$ are independent ${\rm Exp}(1)$ random variables, then $\frac{X}{X+Y}$ is a ${\rm Unif}\left([0,1]\right)$ random variable;
\item If $U$ is a ${\rm Unif}\left([-\pi,\pi]\right)$ random variable, then $\frac{1}{2}(1- \cos{(U)})$ is a ${\rm Beta}\left(\frac{1}{2},\frac{1}{2}\right)$ random variable.
 \end{itemize}
\end{proofof}

\bibliographystyle{alpha}

\bibliography{biblio}

\end{document}